\documentclass[11pt]{amsproc}
\usepackage[mathcal]{euscript}
\usepackage{amsthm}
\usepackage{amsfonts}
\usepackage{amsmath,amssymb}
\usepackage{indentfirst,latexsym,bm,amsthm,graphicx,colortbl}
\usepackage{fancyhdr}
\usepackage{times}
\usepackage{amsmath,amsfonts,amssymb,amsthm}
\usepackage{epsfig}
\usepackage[all]{xy}
\usepackage{tikz}
\usepackage{color}
\usepackage{enumerate}
\usepackage{verbatim}
\newtheorem{theorem}{Theorem}[section]
\newtheorem{lemm}[theorem]{Lemma}
\newtheorem{prop}[theorem]{Proposition}

\theoremstyle{definition}
\newtheorem{defi}[theorem]{Definition}
\newtheorem{example}[theorem]{Example}

\newtheorem{coro}[theorem]{Corollary}
\theoremstyle{remark}
\newtheorem{remark}[theorem]{Remark}

\allowdisplaybreaks

\numberwithin{equation}{section}

\def \<{\langle}
\def \>{\rangle}

\def \g{{\frak{g}}}

\def \be{\begin{equation}\label}
\def \ee{\end{equation}}
\def \bex{\begin{example}\label}
\def \eex{\end{example}}
\def \bl{\begin{lem}\label}
\def \el{\end{lem}}
\def \bt{\begin{thm}\label}
\def \et{\end{thm}}
\def \bp{\begin{prop}\label}
\def \ep{\end{prop}}
\def \br{\begin{rem}\label}
\def \er{\end{rem}}
\def \bc{\begin{coro}\label}
\def \ec{\end{coro}}
\def \bd{\begin{de}\label}
\def \ed{\end{de}}

\begin{document}
\title[Level -1/2 realization of quantum N-toroidal algebras]
{Level $-1/2$ realization of quantum N-toroidal algebras in type $C_n$}

\author[Jing]{Naihuan Jing}
\address{Department of Mathematics, Shanghai University, Shanghai 200444, China\newline \indent
Department of Mathematics, North Carolina State University,
   Raleigh, NC 27695, USA \\}
\email{jing@math.ncsu.edu}

\author[Wang]{Qianbao Wang}
\address{Department of Mathematics, Shanghai University,
Shanghai 200444, China}

\author[Zhang]{Honglian Zhang$^\star$}
\address{Department of Mathematics, Shanghai University,
Shanghai 200444, China} \email{hlzhangmath@shu.edu.cn}



\begin{abstract}
We construct a level $-\frac{1}{2}$ vertex representation of the quantum N-toroidal algebra for type $C_n$, which
is a natural generalization of the usual quantum toroidal algebra. The construction also provides a vertex representation of the quantum toroidal algebra for type $C_n$ as a by-product.
\end{abstract}

\maketitle

\section{\textbf{Introduction}}

 Let $\mathfrak{g}$ be a finite-dimensional complex simple Lie algebra, the $N$-toroidal Lie algebra $\mathfrak{g}_{N, tor}$ associated to $\mathfrak{g}$ is the universal central extension of the multi-loop Lie algebra
 $\mathfrak{g}\otimes \mathbb{C}[t_1^{\pm}, \cdots, t_{N}^{\pm}]$, which
 generalizes both the affine Lie algebra and toroidal Lie algebra. 
 The quantum group $U_q(\mathfrak{g})$ was introduced by Drinfeld \cite{D1, D2} and Jimbo \cite{Jb} independently as a $q$-deformation of the universal enveloping algebra $U(\mathfrak{g})$ of the Lie algebra $\mathfrak{g}$. The quantum group $U_q(\hat{\mathfrak{g}})$ associated to the affine Lie algebra $\hat{\mathfrak{g}}$ is also called the quantum affine algebra, whose representation theory is very rich including vertex representations and finite-dimensional representations and so on. For example, Chari and Pressely classified the finite-dimensional representations (c.f. \cite{CP1}-\cite{CP4}) in terms of Drinfeld polynomials and affine Hecke algebras. The vertex representation was first constructed by Frenkel and Jing  \cite{FJ} for simply-laced types. Subsequently vertex representations for quantum twisted affine algebras were obtained by Jing \cite{J1}.
Bosonic realizations of the quantum affine algebras in other types were constructed in \cite{Be}, \cite{JKM1},\cite{JKM2},\cite{J2}, \cite{JKK}, \cite{JM1}, \cite{JM2} and \cite{J3} etc.

Quantum toroidal algebras $U_q({\mathfrak g}_{tor})$ were introduced in \cite{GKV} through geometric realization related to Langlands reciprocity for algebraic surfaces.
Subsequently, Varagnolo and Vasserot \cite{VV} obtained the Schur-Weyl duality between the representation of the quantum toroidal algebra  $U_q({\mathfrak g}_{tor})$ and the elliptic Cherednik algebra in type $A$.
  Since there exists another two-parameter deformation of the quantum toroidal algebra  $U_q({\mathfrak g}_{tor})$  in type $A$,  there are special interest in the representations of $U_q({\mathfrak g}_{tor})$ see for example \cite{FJW, STU, GJ, H1,H2,  S, FJMM1, FJMM2, GTL, GM}, \cite{M1}-\cite{M4} and the references therein.  Unlike the quantum affine case, the quantum toroidal algebra  $U_q({\mathfrak g}_{tor})$ is not a quantum Kac-Moody algebra, but rather a quantum affnization. Thus these references were focused in type $A$ or the simply-laced cases.  Nevertheless there are still lots of unknowns for the quantum toroidal algebras in type A, and the knowledge on representation theory
  in other types is very limited. In the present paper, we will construct a level-$(-1/2)$ vertex representation of the quantum toroidal algebra for the symplectic type. Actually, our motivation is the recent joint paper \cite{GJXZ}, in which quantum $N$-toroidal algebras (denoted by $U_q({\mathfrak g}_{N, tor})$) were studied as a natural generalization of quantum toroidal algebras $U_q({\mathfrak g}_{tor})$. 
  In \cite{GJXZ}, the quantum $N$-toroidal algebras are shown to be quotients of extended quantized GIM algebras of $N$-fold affinization. In this paper, we will give a level $-\frac{1}{2}$ vertex representation of the quantum $N$-toroidal algebra for type $C_n$, based on the method of \cite{J3}.

The paper is organized as follows. In section 2, we review the definition of  the quantum $N$-toroidal algebra for type $C_n$.
We construct the Fock space and vertex operators in section 3.  Then the main result of constructing
a level-$-1/2$ vertex representation of the quantum $N$-toroidal algebra in type $C_n$ is given. In the last section, we verify the quantum algebra relations to show that the above construction is a realization in detail.

\section{\textbf{Quantum toroidal algebras $U_q(\mathfrak{g}_{N,tor})$}}

In this paper, we always assume  that $\g$ is the finite dimensional simple Lie algebra  of type $C_n$. In this section, we review the definition of quantum N-toroidal algebra $U_q(\mathfrak{g}_{N,tor})$ for the symplectic type recently given in \cite{GJXZ}. For this we recall the data of the simple Lie algebra, affine Lie algebra and toroidal Lie algebra of type $C_n$.

Let $I=\{0,\cdots, n\}$ and $I_0=\{1,\cdots, n\}$.
We denote that $A=(a_{ij}) (i,j\in I_0)$ is the Cartan matrix of $\g$  and $\mathfrak{h}$ is the Cartan subalgebra.   Let $\varepsilon_1,\cdots, \varepsilon_n$ denote the usual orthonormal basis of the Euclidean space $\mathbb{R}^n$. The root system $\Phi$ for $\g$ is $\{\pm(\varepsilon_i\pm\varepsilon_j),\pm2\varepsilon_i|i\neq j\}$ and a base for $\Phi$ is $\Delta=\{\alpha_i|i=1,\cdots, n\}$, where $\alpha_i=\varepsilon_i-\varepsilon_{i+1}$ for $i=1, \cdots, n-1$, $\alpha_n=2\varepsilon_n$. Denote the dominant weights by $\lambda_i=\varepsilon_1+\cdots+\varepsilon_i \ (i=1,\cdots, n)$ and the weight lattice  $P=\mathbf{Z}\varepsilon_1+\cdots +\mathbf{Z}\varepsilon_n$. \ Let $\hat{\mathfrak{g}}$ be the affine Kac-Moody Lie algebra of type $C_n$ with the Cartan subalgebra $\hat{\mathfrak{h}}$ associated to the simple Lie algebra $\mathfrak{g}$. Let $\delta$ be the primitive imaginary root  of  $\hat{\mathfrak{g}}$. \  Let $\alpha_0=\delta-(2\alpha_1+\cdots +2\alpha_{n-1}+\alpha_n)$, then the set of simple roots for $\hat{\g}$ is $\hat{\Delta} =\{\alpha_0,\cdots, \alpha_n\}$. \ We fix the nondegenerate symmetric bilinear form $(\cdot|\cdot)$ on the dual space $\hat{\mathfrak{h}}^*$ 
such that $(\varepsilon_s|\varepsilon_t)=\frac{1}{2}\delta_{st}$ for $s, t\in I_0$, then
\begin{eqnarray*}
 (\alpha_i|\alpha_j)=d_ia_{ij},  \quad (\delta|\alpha_i)=(\delta,\delta)=0    \quad \textrm{for \ all}\ i,j \in I
\end{eqnarray*}
where $(d_0, d_1,\cdots, d_n)=(1,\frac{1}{2}, \cdots, \frac{1}{2}, 1)$ and $\hat{A}=(a_{ij}) (i, j\in I)$ is the Cartan matrix of $\hat{\mathfrak{g}}$.

Suppose $q$ is not a root of unity. Let $q_i =q^{d_i}$ and $[k]_i=\frac{q_i^k-q^{-k}_i}{q_i-q^{-1}_i}$. \ Let $J = \{1,\cdots,N-1\}, \underline{k}= (k_1, k_2,\cdots, k_{N-1})\in \mathbb{Z}^{N-1}$ ,\ $e_s = (0,\cdots,0,1,0,\cdots,0)$ the $s$th standard unit vector of $(N-1)$-dimension lattice $\mathbb Z^{N-1}$
and ${0}$ the $(N-1)$-dimensional zero vector. \
Now we turn to the definition of the quantum $N$-toroidal algebra for type $C_n$ denoted by $U_q(\mathfrak{g}_{N, tor})$ introduced in \cite{GJXZ}.\begin{defi}\label{defi1}
\,The quantum $N$-toroidal algebra $ U_q(\frak{g}_{N, tor})$ is an associative algebra over $\mathbb{F}$ generated by $x_{i}^{\pm}(\underline{k}),\, a_i^{(s)}(r),\, K_i^{\pm}$ and $\gamma_s^{\pm\frac{1}{2}}$  $(i\in I_0,\, s\in J,\,\underline{k}\in\mathbb{Z}^{N-1},\, r\in \mathbb{Z}\backslash \{0\})$, satisfying the relations as follows:
\begin{eqnarray}\label{n:tor1}
&&\gamma_s^{\pm\frac{1}2} \textrm{are central such that}~~\gamma_s^{\pm\frac{1}2}\gamma_s^{\mp\frac{1}2}=1, K_i^{\pm 1}\,K_i^{\mp 1}=1,\\ \label{n:tor2}
&&  K_i^{\pm1}~~ \textrm{and}~~  a_j^{(s)}(r) ~~\textrm{ commute each other,}\\\label{n:tor3}
&&[\,a_i^{(s)}(r),a_j^{(s')}(l)\,]
=\delta_{s,s'}\delta_{r+l,0}\frac{[r\,a_{ij}\,]_i}{r}
\frac{\gamma_s^{r}-\gamma_s^{-r}}{q_j-q_j^{-1}},\\\label{n:tor4}
&&K_ix_{j}^{\pm}(\underline{k})K_i^{-1}=q^{\pm a_{ij}}_ix_{j}^{\pm}(\underline{k}),\\\label{n:tor5}
&&[x_i^{\pm}(ke_s),x_{i}^{\pm}(le_{s'})\,]=0, \qquad \hbox{for}\quad s\neq s' \quad \hbox{and}\quad kl\neq 0,\\\label{n:tor6}
&&[\,a_i^{(s)}(r),x_{j}^{\pm}(\underline{k})\,]=\pm \frac{[\,ra_{ij}\,]_i}{r}
\gamma_s^{\mp\frac{|r|}{2}}x_{j}^{\pm}(re_s{+}\underline{k}),\\\label{n:tor7}
&&[\,x_{i}^{\pm}((k+1)e_s),\,x_{j}^{\pm}(le_s)\,]_{q_i^{\pm a_{ij}}}+
[\,x_{j}^{\pm}((l+1)e_s),\,x_{i}^{\pm}(ke_s)\,]_{q_i^{\pm a_{ij}}}=0,\\ \label{n:tor8}
&&[\,x_{i}^{+}(ke_{s}),\,x_j^{-}(le_{s})\,]=\delta_{ij}\big(\frac{\gamma_s^{\frac{k-l}{2}}\phi_i^{(s)}((k+l))
-\gamma_s^{\frac{l-k}{2}}\psi_i^{(s)}((k+l))}{q_{i}-q_{i}^{-1}}\big),
\end{eqnarray}
where $\phi_i^{(s)}(r)$ and $\psi_i^{(s)}(-r)\, (r\geq 0)$ such that $\phi_i^{(s)}(0)=K_i$ and  $\psi_i^{(s)}(0)=K_i^{-1}$
are defined by:
\begin{gather*}\sum\limits_{r=0}^{\infty}\phi_i^{(s)}(r) z^{-r}_s=K_i \exp \Big(
(q_i-q_i^{-1})\sum\limits_{\ell=1}^{\infty}
a_i^{(s)}(\ell)z^{-\ell}\Big), \\
\sum\limits_{r=0}^{\infty}\psi_i^{(s)}(-r) z^{r}_s=K_i^{-1}\exp
\Big({-}(q_i-q_i^{-1})
\sum\limits_{\ell=1}^{\infty}a_i^{(s)}(-\ell)z^{\ell}\Big),
\end{gather*}
\begin{eqnarray}\label{n:tor9}
 && Sym_{{k_1}, \cdots, k_m}\sum_{l=0}^{m=1-a_{ij}}(-1)^l\Big[{m\atop  l}\Big]_{i}x_i^{\pm}(k_1e_s)\cdots
x_i^{\pm}(k_le_s)x_{j}^{\pm}(\ell e_s)\\
&&\hspace{2.65cm}x_i^{\pm}(k_{l+1}e_s)\cdots x_i^{\pm}(k_me_s)=0,
\qquad\hbox{for } \ i \neq j, \nonumber\\\label{n:tor10}
&&\sum_{k=0}^{3}(-1)^k
	\Big[{3\atop  k}\Big]_{i}x_i^{\pm}({e_sm_1})\cdots x_i^{\pm}({e_sm_{k}}) x_{i}^{\mp}({e_{s'}\ell})x_i^{\pm}({e_sm_{k+1}})\cdots x_i^{\pm}({e_sm_{3}})=0,\\ \nonumber
	&&\hspace{3.1cm} \qquad	~~~~\hbox{for}~~~~  i\in I_0~~~~ \hbox{and}~~~~m_1m_2m_3\ell\neq 0, \, s\neq s'\in J,
\end{eqnarray}
where the $q$-bracket is defined as $[a, b]_{u}\doteq ab-uba$  and	$\textit{Sym}_{{m_1},\cdots,
		{m_{n}}}$  denotes the symmetrization with respect to the indices $({m_1},\cdots,
		{m_{n}})$.
\end{defi}


\begin{remark}
In the case of $N=2$, the quantum N-toroidal algebras are just the quantum toroidal algebras \cite{GKV}. Therefore the former are
 natural generalizations of the usual quantum toroidal algebra, \  just like  $N$-toroidal Lie algebras vs. the $2$-toroidal Lie algebras.
\end{remark}

\begin{remark}
For a fixed $s\in J$, the subalgebra $U_q^{(s)}$ of $U_q(\mathfrak{g}_{N,tor})$ generated by the elements $x_{i}^{\pm}(ke_s),a_i^{(s)}(r),K_i^{\pm 1}, \  \gamma_s^{\pm \frac{1}{2}}$ for $i\in I$ 
isomorphic to the quantum 2-toroidal algebra defined in \cite{GKV}.
\end{remark}

\begin{remark}\,For the formal variables $\underline{z}=(z_1, \cdots,\, z_{N-1})$, denote $\underline{z}^{\underline{k}}=\prod\limits_{s=1}^{N-1} z_s^{k_s}$. We set the generating functions of formal variables for $i\in I_0$ and $s\in J$ as follows,
\begin{gather*}
\delta(z)=\sum_{k\in\mathbb{Z}}z^{k},\qquad
  x_{i}^{\pm}(\underline{z})=\sum_{\underline{k} \in \mathbb{Z}^{N-1}}x_{i}^{\pm}(\underline{k}) \underline{z}^{-\underline{k}},\\
     x_{i,s}^{\pm}({z})=\sum_{k\in \mathbb{Z}}x_{i}^{\pm}(ke_s) z^{-k},\\
\phi_i^{(s)}(z) =\sum_{m \in \mathbb{Z}_+}\phi_i^{(s)}(m) z^{-m}, \qquad
\psi_i^{(s)}(z)  = \sum_{n \in \mathbb{Z}_+}\psi_i^{(s)}(-n) z^{n}.
\end{gather*}

It is not difficult to see that relations from \eqref{n:tor5} to \eqref{n:tor10} are equivalent to the following relations, respectively,
\begin{eqnarray}\label{n:tor5-1}
&&\lim_{z\to w}[{x}_{i,s}^{\pm}(z), {x}_{i,s'}^{\pm}(w)]=0, \qquad \hbox{for} \quad s\neq s',\\\label{n:tor6-1}
&&\psi_i^{(s)}(z)x_j^{\pm}(\underline{w})\psi_i^{(s)}(z)^{-1}
=g_{ij}\Bigl(\frac{z}{w_s}\gamma_s^{\mp
\frac{1}{2}}\Bigr)^{\pm1}x_j^{\pm}(\underline{w}),  \\ \nonumber
&&\phi_i^{(s)}(z)x_j^{\pm}(\underline{w})\phi_i^{(s)}(z)^{-1}=g_{ij}\Bigl(\frac{w_s}{z}\gamma_s^{\mp
\frac{1}{2}}\Bigr)^{\mp1}x_j^{\pm}(\underline{w}),\\
\label{n:tor7-1}
&&(z-q_i^{\pm a_{ij}}w)\,x_{i,s}^{\pm}(z)x_{j,s}^{\pm}(w)+(w-q_i^{\pm a_{ij}}z)\,x_{j,s}^{\pm}(w)\,x_{i,s}^{\pm}(z)=0,\\
\label{n:tor8-1}
&&[\,x_{i,s}^+(z),
x_{j,s}^-(w)\,]=\frac{\delta_{ij}}{(q_i-q^{-1}_i)zw}\Big(\phi_i^{(s)}(w\gamma_s^{\frac{1}2})\delta(\frac{w\gamma_s}{z})
-\psi_i^{(s)}(w\gamma_s^{-\frac{1}{2}})\delta(\frac{w\gamma^{-1}_s}{z})\Big),\\\label{n:tor9-1}
&&Sym_{{z_1},\cdots
		{z_{n}}}\sum_{k=0}^{n=1-a_{i{j}}}(-1)^k
	\Big[{n\atop  k}\Big]_{i}x_{i,s}^{\pm}({z_1})\cdots x_{i,s}^{\pm}({z_k}) x_{j,s}^{\pm}({w})\\\nonumber
	&&\qquad  \hspace{5cm}    \times x_{i,s}^{\pm}({z_{k+1}})\cdots x_{i,s}^{\pm}({z_{n}})=0,
	\quad\hbox{for} \quad   i\neq j \\\label{n:tor10-1}
&& \lim_{z_i\to w}\sum_{k=0}^{3}(-1)^k
	\Big[{3\atop  k}\Big]_{i}{x}_{i,s}^{\pm}(z_1)\cdots {x}_{i,s}^{\pm}(z_{k}){x}_{i,s'}^{\mp}(w) {x}_{i,s}^{\pm}(z_{k+1})\cdots{x}_{i,s}^{\pm}(z_{3})=0,\\\nonumber
&&\hspace{10.65cm}	~~~\hbox{for} ~~~i\in I~~~ \hbox{and}~~~ s\neq s'\in J.
\end{eqnarray}
where
$g_{ij}(z):=\sum_{n\in
\mathbb{Z}_+}c_{ijn}z^{n}$ is the Taylor series expansion of $g_{ij}(z)=\frac{zq_i^{a_{ij}}-1}{z-q_i^{a_{ij}}}$ at $z=0$ in $\mathbb{C}$.
\end{remark}

\section{\textbf{Vertex representations}}
In this section, we construct the Fock space and obtain a level-$(-1/2)$ vertex representation of quantum $N$-toroidal algebra $U_q(\mathfrak{g}_{N,tor})$ for type $C_n$, based on the method in \cite{JKM1}.

First of all,\ let us introduce the quantum Heisenberg algebra $U_q(\mathfrak{h}_{N,tor})$,\ which is generated by $a_i^{(s)}(r),b_i^{(s)}(r)$ for $i\in I, s\in J$ satisfying the following relations:
\begin{gather} \label{12}
  [a_i^{(s)}(r),a_j^{(s')}(t)]=\delta_{ss'}\delta_{r+t,0}\dfrac{[ra_{ij}]_i}{r}\dfrac{q^{-r/2}-q^{r/2}}{q_j-q_j^{-1}},\\
  [b_i^{(s)}(r),b_j^{(s')}(t)]=r\delta_{ss'}\delta_{ij}\delta_{r+t,0},\\
  [a_i^{(s)}(r),b_j^{(s')}(t)]=0.
\end{gather}

 Let $S(\mathfrak{h}^-_{N,tor})$ be the symmetric algebra generated by $a_i^{(s)}(-l),b_i^{(s)}(-l)$ with $l$ being a positive integer. Then $S(\mathfrak{h}^-_{N, tor})$ is a  $U_q(\mathfrak{h}_{N,tor})$-module by letting $a_i^{(s)}(-l),b_i^{(s)}(-l)$
 act as multiplication operators and $a_i^{(s)}(l),b_i^{(s)}(l)$
 operate as differentiation subject the Heisenberg algebra relations.
 Let $\tilde{P}$ be the affine weight lattice $\tilde{P}=\mathbb Z\lambda_0+\cdots+\mathbb Z\lambda_n$, and set $\tilde{P}'\simeq \tilde{P}$, an identical copy of $\tilde{P}$ corresponding to $b_i's$. 
 We define the Fock space
 \begin{equation*}
 \mathcal{F}=S(\mathfrak{h}^-_{N,tor})\otimes \mathbf{C}[\tilde{P}]\otimes\mathbf{C}[\tilde{P}']\otimes\mathbf{C}[\mathbb ZJ],
 \end{equation*}
where $\mathbf{C}[G]$ is the group algebra of the abelian group $G$.

We equip the lattice $\mathbb ZJ$ with the nondegerate bilinear form $(\ |\ )$ defined by
\begin{equation}
(s_i|s_j)=\begin{cases} -1 & i\neq j\\ 0 & i=j \end{cases}.
\end{equation}
where we list the elements of $J$ as $\{s_1, s_2, \cdots, s_{N-1}\}$.

The action of operators $e^{a_i}$, $e^{b_i}$ , $e^{s_i}$, $a_i^{(s)}(m),b_i^{(s)}(m)$, $s_i(0)$ on ${\mathcal{F}} $ is defined by the following relations,
\begin{gather*}
  e^{a_i}.e^\alpha \otimes e^\beta\otimes e^{s}=e^{\alpha_i+\alpha}\otimes e^\beta\otimes e^{s}, \quad \quad e^{b_i}.e^\alpha \otimes e^\beta\otimes e^{s}=e^\alpha \otimes e^{\varepsilon_i+\beta}\otimes e^{s},
  \\ z^{a_i(0)}.(e^\alpha\otimes e^\beta\otimes e^{s})=z^{(\alpha_{i}|\alpha)}(e^\alpha\otimes e^\beta\otimes e^{s}),\quad  z^{b_i(0)}.(e^\alpha\otimes e^\beta\otimes e^{s}) =z^{(2\varepsilon_i|\beta)}(e^\alpha\otimes e^\beta\otimes e^{s}),\\
  q^{s'}.(e^\alpha\otimes e^\beta\otimes e^{s})=q^{(s'|s)}(e^\alpha\otimes e^\beta\otimes e^{s})
\end{gather*}
where $\alpha\in \tilde{P},\beta \in \tilde{P}', s, s'\in\mathbb ZJ$. Here we have added $\varepsilon_0$ such that $(\varepsilon_i|\varepsilon_j)=\frac12\delta_{ij}$
for $i, j\in I$.

The normal order $: \quad :$ is defined as usual,
\begin{gather*}
:a_i^{(s)}(r)a_j^{(s)}(t):=
\begin{cases}
         a_i^{(s)}(r)a_j^{(s)}(t),  & \mbox{if }  r < t; \\
        a_j^{(s)}(t)a_i^{(s)}(r), & \mbox{if }  r\geq t,
         \end{cases}\\
  \quad :e^{a_j}z^{a_i(0)}:=:z^{a_i(0)} e^{a_j}:= e^{a_j}z^{a_i(0)},
\end{gather*}
and
\begin{gather*}
  :b_i^{(s)}(r)b_j^{(s)}(t):=
  \begin{cases}
         b_i^{(s)}(r)b_j^{(s)}(t),  & \mbox{if }  r < t; \\
        b_j^{(s)}(t)b_i^{(s)}(r), & \mbox{if }  r\geq t,
         \end{cases}\\
  \quad :e^{b_j}z^{b_i(0)}:=:z^{b_i(0)}e^{b_j}:=e^{b_j}z^{b_i(0)}.\\
  \quad :e^{s}q^{s'(0)}:=:q^{s'(0)}e^{s}:=e^{s}q^{s'(0)}.
\end{gather*}
\begin{prop}
 We  have the relations between the operators $a_{i}(0), b_j(0), s, e^{a_i}, e^{b_j}, q^{s'} $ as follows
 \begin{equation*}
  [a_i(0),e^{a_j}]=(\alpha_i|\alpha_j)e^{a_j},\quad [b_i(0),e^{b_j}]=2(\varepsilon_i|\varepsilon_j)e^{b_j}, \quad
  [s(0),q^{s'}]=(s|s')q^{s'}.
 \end{equation*}
\end{prop}

Now we define the following vertex operators for $i\in I$,
\begin{gather*}
 \begin{split}
   &X_{i}^\pm(\underline{z})=\exp\left(\pm\sum\limits_{s=1}^{N-1}\sum\limits_{k_s=1}^\infty\dfrac{a_i^{(s)}(-k_s)}{[k_s/d_i]_i}q^{\mp k_s/2}z_s^{k_s}\right)\exp\left(\mp\sum\limits_{s=1}^{N-1}\sum\limits_{k_s=1}^\infty\dfrac{a_i^{(s)}(k_s)}{[k_s/d_i]_i}q^{\mp k_s/2}z_s^{-k_s}\right)\\
   & \hskip2.5cm\times e^{\pm \alpha_i}\prod_{s=1}^{N-1}z_s^{\pm a_i(0)+1},\\
   &Y^\pm _{i,s}({z})=\exp\Big(\pm \sum_{k=1}^{\infty}\dfrac{a_{i}^{(s)}(-k)}{[-(1/2d_i)k]_i}q^{\pm k/4}z^{k}\Big)\exp\Big(\mp \sum_{k=1}^{\infty}\dfrac{a_{i}^{(s)}(k)}{[-(1/2d_i)k]_i}q^{\pm k/4}z^{-k}\Big)\\
 &\hskip2.5cm \times e^{\pm  a_i} z^{\mp 2a_{i}(0)},\\
   &Z_{i,s}^\pm ({z})=\exp\Big(\pm\sum_{k=1}^{\infty}\dfrac{b_{i}^{(s)}(-k)}{k}z^{k}  \Big) \exp  \Big( \mp\sum_{k=1}^{\infty}\dfrac{b_{i}^{(s)}(k)}{k}z^{-k}\Big)
 e^{\pm b_i}z^{\pm b_i(0)}.
\end{split}
 \end{gather*}

 For simplicity, we introduce the following notations for $\epsilon=\pm 1$ or $\pm$, $i=1, \cdots, n-1$, $j=0, n$ and $s\in J$,
\begin{gather*}
 X_{i\epsilon,s}^+({z})=Z_{i,s}^+(q^{\epsilon/2}{z})Z_{i+1,s}^-({z})Y_{i,s}^+({z})e^{s}q^{\epsilon s(0)},\qquad
 X_{i\epsilon,s}^-({z})=Z_{i,s}^-({z})Z_{i+1,s}^+(q^{\epsilon/2}{z})Y_{i,s}^-({z})e^{-s}q^{-\epsilon s(0)},\\
X_{j\epsilon,s}^+({z})=:Z_{j,s}^+(q^{\frac{1+\epsilon}{2}}{z})Z_{j,s}^+(q^{\frac{-1+\epsilon}{2}}{z})Y_{j,s}^+({z}):e^{s}q^{2\epsilon s(0)},\qquad
X_{j0,s}^+({z})=:Z_{j,s}^+(q{z})Z_{j,s}^+(q^{-1}z)Y_{j,s}^+({z}):e^s.
 \end{gather*}

Now we give the main result of the paper.
\begin{theorem}\label{main}
   For $i\in I$ and $s\in J$,  the  Fock
   space $\mathcal F$ is a $ U_q(\mathfrak{g}_{N,tor})$-module  for type $C_n$ of level $-\frac{1}{2}$ under the action $\psi$ defined by :
   $$\begin{array}{rcl}
   \gamma_s^{\pm1}&\mapsto& q^{\mp1/2},\\
    K_i &\mapsto&  q_i^{a_i(0)},  \\
   x_{i}^{\pm}(\underline{{z}})&\mapsto&X_{i}^{\pm}(\underline{{z}}),\\
   x_{i,s}^{\pm}({z})&\mapsto&X_{i,s}^{\pm}({z}), \\
   \phi_i^{(s)}({z}) &\mapsto&\Phi_i^{(s)}(z) ,\\
   \psi_i^{(s)}({z}) &\mapsto& \Psi_i^{(s)}(z),
   \end{array}$$
where  $\Phi_i^{(s)}(z)$ and $\Psi_i^{(s)}(z)$ \ are defined the same way as  \ $\phi_i^{(s)}(z)$ and $\psi_i^{(s)}(z)$ respectively, \ $i=1, \cdots, n-1$,  and $j=0, n$,
\begin{gather*}
 \begin{split}
 &X_{i,s}^{\pm}({z})=\dfrac{1}{(q^{1/2}-q^{-1/2})z} (X_{i+,s}^\pm({z})-X_{i-,s}^\pm({z})),\\
&X_{j,s}^+({z})=-\frac{q^{1/2} X_{j+,s}^+({z})+q^{-1/2}X_{j-,s}^+({z})-[2]_1X_{j0,s}^+({z})}
  {(q-q^{-1})(q^{1/2}-q^{-1/2})z^2}, \\
  &   X_{j,s}^-({z})= :Z_{j,s}^-(q^{1/2}{z})Z_{j,s}^-(q^{-1/2}{z}):Y_{j,s}^-({z})e^{-s}q^{-s(0)}.
\end{split}
 \end{gather*}
\end{theorem}
\section{\textbf{Proof of Theorem \ref{main}}}
In this section, we proceed to prove Theorem \ref{main} in detail. First of all, let us give some relations that will be used in the sequel.

 \begin{lemm}\label{lemm1} The following relations for $Y_{i,s}^{\pm} ({z})$ and $Z_{i,s}^{\pm} ({z})$ holds,


 \begin{multline}\label{1}
  Y_{i,s}^{\pm} ({z}) Y_{j,s}^\pm ({w})=:Y_{i,s}^{\pm} ({z}) Y_{j,s}^{\pm} ({w}):  \qquad \\
 \times \begin{cases}
         1,  &\makebox[60pt][r]{$\text{if \ } (\alpha_i| \alpha_j)=0$}\\
           (z-q^{\pm 1/2}w), &\makebox[60pt][r]{$\text{if \ }(\alpha_i| \alpha_j)=-\frac{1}{2}$} \\
          \big((z-q^{\pm 1}w)(z-w)\big)^{-(\alpha_i| \alpha_j)}, & \makebox[60pt][r]{$\text{if \ } (\alpha_i| \alpha_j)=\pm 1 $}\\
         \big((z-w)(z-qw)(z-q^{-1}w)(z-q^{\pm2}w)\big)^{-1}, &\makebox[60pt][r]{$\text{if \ } (\alpha_i| \alpha_j)=2$}.
        \end{cases}
 \end{multline}

\begin{multline}\label{2}
  Y_{i,s}^\pm ({z}) Y_{j,s}^\mp ({w})=:Y_{i,s}^\pm ({z}) Y_{j,s}^\mp ({w}):    \\
 \times \begin{cases}
          1,   &\makebox[60pt][r]{$\text{if \ } (\alpha_i| \alpha_j)=0$},\\
          (z-w)^{-1},&\makebox[60pt][r]{$\text{if \ }(\alpha_i| \alpha_j)=-\frac{1}{2}$}, \\
            \big((z-q^{-1/2}w)(z-q^{1/2}w)\big)^{(\alpha_i|\alpha_j)}, &
            \makebox[60pt][r]{$\text{if \ } (\alpha_i| \alpha_j)=\pm 1 $}, \\
            (z-q^{-\frac{1}{2}}w)(z-q^{\frac{1}{2}}w)(z-q^{-\frac{3}{2}}w)(z-q^{\frac{3}{2}}w),  &\makebox[60pt][r]{$\text{if \ } (\alpha_i| \alpha_j)=2$}.
         \end{cases}
\end{multline}
 \begin{equation}\label{3}
  Z_{i,s}^{\epsilon } ({z}) Z_{j,s}^{\epsilon'} ({w})=:Z_{i,s}^{\epsilon } ({z}) Z_{j,s}^{\epsilon'} ({w}):(z-w)^{\epsilon\epsilon'\delta_{ij}},
 \end{equation}
where  $(z-w)^{-1}$ is the power series in $w/z$ as follows:
\begin{equation*}
  (z-w)^{-1}=\sum_{k=0}^{\infty}w^kz^{-k-1}. 
\end{equation*}
\end{lemm}
\begin{proof}
Here we only show the relation \eqref{1} in the case of $(\alpha_i|\alpha_j)=-1/2$, other relations can be verified in a similar manner.  According to the definitions of $ Y_{i,s}^\pm ({z})$ and the normal order, we obtain immediately,
\begin{gather*}
\begin{split}
& Y_{i,s}^\pm ({z}) Y_{j,s}^\pm ({w})\\
=&:Y_{i,s}^\pm ({z}) Y_{j,s}^\pm ({w}):
 exp\Big(-\sum\limits_{k=1}^{\infty}\dfrac{q^{\pm k/2}z^{-k}w^{k}}{[-(1/2d_i)k]_i[-(1/2d_j)k]_j}[ a_i^{(s)}(k), a_j^{(s)}(-k)]\Big)z \\
   & =:Y_{i,s}^\pm ({z}) Y_{j,s}^\pm ({w}):exp\Big( - \sum\limits_{k=1}^{\infty}\big(\dfrac{q^{\pm1/2}w}{z}\big)^{k}\dfrac{1}{k}  \Big)z \\
   & =:Y_{i,s}^\pm ({z}) Y_{j,s}^\pm ({w}): (z-{q^{\pm 1/2}w}).
 \end{split}
  \end{gather*}
\end{proof}
\begin{lemm}\label{l4.2}
From the notations of  $X_{i\epsilon,s}^\pm ({z})$, it holds for $i=1,\cdots, n-1$ and $j=0, n$,
\begin{eqnarray}\label{4}
	&& X_{i\epsilon,s}^\pm ({z})X_{i\epsilon',s}^\pm ({w})=
  :X_{i\epsilon,s}^\pm ({z})X_{i\epsilon',s}^\pm ({w}):(q^{\epsilon/2}z-q^{\epsilon'/2}w)(z-q^{\pm 1}w)^{-1},\\\label{5}
&&  X_{i\epsilon,s}^+ ({z})X_{i\epsilon',s}^- ({w})=
  :X_{i\epsilon,s}^+ ({z})X_{i\epsilon',s}^- ({w}):(q^{\epsilon/2}z-w)^{-1}(z-q^{-\epsilon'/2}w), \\
  \label{6}
&&   X_{i\epsilon,s}^- ({z})X_{i\epsilon',s}^+ ({w})=
 :X_{i\epsilon,s}^- ({z})X_{i\epsilon',s}^+ ({w}):(z-q^{-\epsilon'/2}w)(q^{\epsilon/2}z-w)^{-1},\\
 \label{7}
&&   X_{i\epsilon,s}^+ ({z})X_{(i+1)\epsilon',s}^+ ({w})=:X_{i\epsilon,s}^+ ({z})X_{(i+1)\epsilon',s}^+ ({w}):
   (z-q^{\epsilon'/2}w)^{-1}(z-q^{1/2}w),\\
\label{8}
&&   X_{i\epsilon,s}^- ({z})X_{(i+1)\epsilon',s}^+ ({w})
   =X_{(i+1)\epsilon',s}^+ ({w})X_{i\epsilon,s}^- ({z})\\\notag
&& \hspace{3cm}     =:X_{i\epsilon,s}^- ({z})X_{(i+1)\epsilon',s}^+ ({w}):(q^{\epsilon/2}z-q^{\epsilon'/2}w)(z-w)^{-1},\\
    \label{9}
   && X_{i\epsilon,s}^+ ({z})X_{(i+1)\epsilon',s}^- ({w})=X_{(i+1)\epsilon',s}^-({w})X_{i\epsilon,s}^+({z})=:X_{i\epsilon,s}^+ ({z})X_{(i+1)\epsilon',s}^- ({w}):,
     \end{eqnarray}

\begin{eqnarray}\label{10}
 && X_{(i+1)\epsilon,s}^+ ({z})X_{i\epsilon',s}^- ({w})=
  :X_{(i+1)\epsilon,s}^+ ({z})X_{i\epsilon',s}^- ({w}):(q^{\epsilon/2}z-q^{\epsilon'/2}w)(z-w)^{-1}, \\
\label{11}
 &&  X_{(i+1)\epsilon,s}^+ ({z})X_{i\epsilon',s}^+ ({w})=
   : X_{(i+1)\epsilon,s}^+ ({z})X_{i\epsilon',s}^+ ({w}):(q^{\epsilon/2}z-w)^{-1}(z-q^{1/2}w),\\
 \label{12}
 && X_{j\epsilon,s}^+({z})X_{j,s}^-({w})=:X_{j\epsilon,s}^+({z})X_{j,s}^-({w}):
  \Big(\dfrac{q^{-3\epsilon /2}(z-q^{3\epsilon/2}w)}{q^{\epsilon/2}z-w}\Big)^{|\epsilon|},\\
  \label{13} &&   X_{j,s}^-({w})X_{j\epsilon,s}^+({z})=:X_{j,s}^-({w})X_{j\epsilon,s}^+({z}):
  \dfrac{w-q^{-3\epsilon/2}z}{w-q^{\epsilon/2}z},  \\
\label{14}
 && X_{n\epsilon,s}^+({z})X^-_{(n-1)\epsilon',s}({w})=  X^-_{(n-1)\epsilon',s}({w})X_{n\epsilon,s}^+({z})\\\notag
 && \hspace{2cm} = :X_{n\epsilon,s}^+({z})X^-_{(n-1)\epsilon',s}({w}): \dfrac{(q^{1+\epsilon}_\epsilon  z-q^{\epsilon'/2}w)(q^{-1+\epsilon}_\epsilon z-q^{\epsilon'/2}w)}{(z-q^{-1/2}w)(z-q^{1/2}w)},\\
\label{15} &&     X_{0\epsilon,s}^+({z})X^-_{1\epsilon',s}({w})=
    X^-_{1\epsilon',s}({w})X_{0\epsilon,s}^+({z})\\\notag
 && \hspace{2cm} = :X_{0\epsilon,s}^+({z})X^-_{1\epsilon',s}({w}):
   \dfrac{1}{(z-q^{-1/2}w)(z-q^{1/2}w)}.\\ \label{com1}
 && X^-_{i,s}(z)X^-_{j,s}(w)=:X^-_{i,s}(z)X^-_{j,s}(w):\frac{z-w}{z-q^{-2}w}, \\ \label{com2}
 && X^{\pm}_{i\epsilon, s}(z)X^{\pm}_{i\epsilon', s'}(w)=:X^{\pm}_{i\epsilon, s}(z)X^{\pm}_{i\epsilon', s'}(w):, \quad \mathrm{for}\ s\neq s'\\ \label{com3}
 && X^{\pm}_{i\epsilon, s}(z)X^{\mp}_{i\epsilon', s'}(w)=:X^{\pm}_{i\epsilon, s}(z)X^{\mp}_{i\epsilon', s'}(w):, \quad \mathrm{for}\ s\neq s'\\
 && X^{+}_{j\epsilon, s}(z)X^{+}_{j\epsilon', s'}(w)=:X^{+}_{j\epsilon, s}(z)X^{+}_{j\epsilon', s'}(w):, \quad \mathrm{for}\ s\neq s'\\
 && X^{-}_{j, s}(z)X^{-}_{j, s'}(w)=:X^{-}_{j, s}(z)X^{-}_{j, s'}(w):q^{-2}, \quad \mathrm{for}\ s\neq s'
\end{eqnarray}
\end{lemm}
\begin{proof} To check relation \eqref{4}, without loss of generality, we only prove it in the case of "+".  Using
	relations \eqref{1} and \eqref{3}, one gets directly that,
\begin{eqnarray*}
X_{i\epsilon,s}^+ ({z})X_{i\epsilon',s}^+ ({w})
&=&Z_{i,s}^+(q^{\epsilon/2}{z})Z_{i+1,s}^-({z})Y_{i,s}^+({z})Z_{i,s}^+(q^{\epsilon'/2}{w})Z_{i+1,s}^-({w})Y_{i,s}^+({w})\\
&=&:X_{i\epsilon,s}^+ ({z})X_{i\epsilon',s}^+ ({w}):(q^{\epsilon/2}z-q^{\epsilon'/2w})(z-w)(z-qw)^{-1}(z-w)^{-1}\\
&=&:X_{i\epsilon,s}^+ ({z})X_{i\epsilon',s}^+ ({w}):(q^{\epsilon/2}z-q^{\epsilon'/2w})(z-qw)^{-1}.
\end{eqnarray*}

Since relations \eqref{5} and \eqref{6} can be verified similarly, we only check relation \eqref{5}. By relation \eqref{2},  we obtain immediately that
\begin{eqnarray*}
&&X^+_{i\epsilon,s}({z})X^-_{i\epsilon',s}({w})\\
&=&:X^+_{i\epsilon,s}({z})X^-_{i\epsilon',s}({w}):(q^{\epsilon/2}z-w)^{-1}(z-q^{\epsilon'/2}w)^{-1}(z-q^{-1/2}w)(z-q^{1/2}w)\\
&=&:X^+_{i\epsilon,s}({z})X^-_{i\epsilon',s}({w}):(q^{\epsilon/2}z-w)^{-1}(z-q^{-\epsilon'/2}w).\end{eqnarray*}

To check relation \eqref{7},  we have that,
\begin{eqnarray*}
X_{i\epsilon,s}^+({z})X_{(i+1)\epsilon',s}^+({w})
&=&Z_{i,s}^+(q^{\epsilon/2}{z})Z_{i+1,s}^-({z})Y_{i,s}^+({z})
Z_{i+1,s}^+(q^{\epsilon'/2}{w})Z_{i+2,s}^-({w})Y_{i+1,s}^+({w})\\
&=&:X_{i\epsilon,s}^+({z})X_{(i+1)\epsilon',s}^+({w}):(z-q^{\epsilon'/2}w)^{-1}(z-q^{1/2}w).
\end{eqnarray*}
Similarly 
one can check relations \eqref{8}-\eqref{11}.

Below we verify the relation \eqref{12} directly, and \eqref{15} can be similarly done,
\begin{eqnarray*}
 &&X_{j\epsilon,s}^+({z})X_{j,s}^-({w})\\
 &=&:Z_{j,s}^+(q_\epsilon^{1+\epsilon}{z})Z_{j,s}^+(q_\epsilon^{-1+\epsilon}{z})Y_{j,s}^+({z})::Z_{j,s}^-(q^{1/2}{w})Z_{j,s}^-(q^{-1/2}{w}):Y_{j,s}^-({w})\\
 &=&:X_{j\epsilon,s}^+({z})X_{j,s}^-({w}):(q_\epsilon^{1+\epsilon}z-q^{1/2}w)^{-1}(q_\epsilon^{1+\epsilon}z-q^{-1/2}w)^{-1}(q_\epsilon^{-1+\epsilon}z-q^{1/2}w)^{-1}\\
 &&\times (q_\epsilon^{-1+\epsilon}z-q^{-1/2}w)^{-1}(z-q^{-1/2}w)(z-q^{1/2}w)(z-q^{-3/2}w)(z-q^{3/2}w).
 \end{eqnarray*}
Then we proceed with verification in 
three subcases: $\epsilon=\pm$ or $0$.  Direct calculation yields that,
 \begin{eqnarray*}
 &&X_{j\epsilon,s}^+({z})X_{j,s}^-({w})\\
&=&:X_{j\epsilon,s}^+({z})X_{j,s}^-({w}):
  \Big(\dfrac{q^{-3\epsilon /2}(z-q^{3\epsilon/2}w)}{q^{\epsilon/2}z-w}\Big)^{|\epsilon|}.
\end{eqnarray*}

For relation \eqref{13} and \eqref{14},  the proofs are similar, \ here we only check relation \eqref{14}.
\begin{eqnarray*}
 && X_{0\epsilon,s}^+({z})X^-_{1\epsilon',s}({w})\\
 &=&:Z_{0,s}^+(q_\epsilon^{1+\epsilon}{z})Z_{0,s}^+(q_\epsilon^{-1+\epsilon}{z})Y_{0,s}^+({z}):Z_{1,s}^-({w})Z_{2,s}^+(q^{\epsilon'/2}{w})Y_{1,s}^-({w}) \\
  & =& :X_{0\epsilon,s}^+({z})X^-_{1\epsilon',s}({w}):
\dfrac{1}{(z-q^{-1/2}w)(z-q^{1/2}w)}\\
&=& X^-_{1\epsilon',s}({w})X_{0\epsilon,s}^+({z}),
\end{eqnarray*}
which completes the proof of lemma \ref{l4.2}.
\end{proof}

With the help of the above two lemmas, we can now prove Theorem \ref{main}. It means that we need to check $\psi$ satisfy all defining relations \eqref{n:tor1}-\eqref{n:tor10}.
It is obvious that relations \eqref{n:tor1}-\eqref{n:tor4} follow from the constructions of the vertex operators. Thus it suffices to show $\psi$ keeps relations \eqref{n:tor5}-\eqref{n:tor10}, or equivalently \eqref{n:tor5-1}-\eqref{n:tor10-1},
which will be explained in more detail as follows.

To show relation \eqref{n:tor5-1}, we first look at the case $i=1, \ldots, n-1$ for example. For $s\neq s'$ 
\begin{eqnarray*}
X_{i,s}^{\pm}(z)X_{i,s'}^{\pm}(w)&=&\frac1{(q^{1/2}-q^{-1/2})^2zw}\sum_{\epsilon, \epsilon'}(-1)^{\epsilon\epsilon'}X_{i\epsilon,s}^{\pm}(z)X_{i\epsilon',s'}^{\pm}(w)\\
&=&\frac1{(q^{1/2}-q^{-1/2})^2zw}\sum_{\epsilon, \epsilon'}(-1)^{\epsilon\epsilon'}:X_{i\epsilon,s}^{\pm}(z)X_{i\epsilon',s'}^{\pm}(w):.
\end{eqnarray*}
Therefore, for $s\neq s'$.
\begin{equation*}
[X_{i,s}^{\pm}(z), X_{i,s'}^{\pm}(w)]=0.
\end{equation*}

Similarly for $s\neq s'$ and $j=0, n$, by Lemma \ref{l4.2}
 we have (in the following $\epsilon, \epsilon'=\pm, 0$)
\begin{eqnarray*}
X_{js}^{+}(z)X_{js'}^{+}(w)&=&\frac{\sum_{\epsilon, \epsilon'}(-1)^{1+\epsilon\epsilon'}[2-|\epsilon|]_1[2-|\epsilon'|]_1X_{j\epsilon,s}^{+}(z)X_{j\epsilon',s'}^{+}(w)}
{(q^{1/2}-q^{-1/2})^2(q-q^{-1})^2z^2w^2}\\
&=&\frac{\sum_{\epsilon, \epsilon'}(-1)^{1+\epsilon\epsilon'}[2-|\epsilon|]_1[2-|\epsilon'|]_1:X_{j\epsilon,s}^{+}(z)X_{j\epsilon',s'}^{+}(w):}
{(q^{1/2}-q^{-1/2})^2(q-q^{-1})^2z^2w^2},
\end{eqnarray*}
which implies that $[X_{j,s}^{+}(z)X_{j,s'}^{+}(w)]=0$ for $s\neq s'$. The remaining cases are the same.

Now we turn to check relation \eqref{n:tor6}, it suffices to verify that for $i\in I, j=0,1,\cdots, n-1$,
$$\Psi_i^{(s)}(z_s)X_{j}^{\pm}(\underline{w})\Psi_i^{(s)}(z_s)^{-1}
=g_{ij}\Bigl(\frac{z_s}{w_s}q^{\mp
\frac{1}{2}}\Bigr)^{\pm1}X_{j}^{\pm}(\underline{w}).$$

Actually, we have that,
\begin{eqnarray*}
&&\Psi_i^{(s)}(z_s)X_{j}^{\pm}(\underline{w})\\
&=&q^{-a_i(0)}\exp\Bigl(-(q_i-q_i^{-1})\sum\limits_{k=1}^{\infty}a_i^{(s)}(-k)z_s^k\Bigr)
\exp\left(\pm\sum\limits_{s=1}^{N-1}\sum\limits_{k_s=1}^\infty\dfrac{a_j^{(s)}(-k_s)}{[k_s/d_j]_j}q^{\mp k_s/2}w_s^{k_s}\right)\\
   &&\hskip0.5cm\times\exp\left(\mp\sum\limits_{s=1}^{N-1}\sum\limits_{k_s=1}^\infty\dfrac{a_j^{(s)}(k_s)}{[k_s/d_j]_j}q^{\mp k_s/2}w_s^{-k_s}\right) e^{\pm \alpha_j}\prod_{s=1}^{N-1}w_s^{\pm a_j(0)+1},\\
&=&\exp\Bigl(\pm(q_i-q_i^{-1})\sum\limits_{k_s=1}^\infty\frac{[a_i^{(s)}(-k_s), a_j^{(s)}(k_s)]}{[k_s/d_j]_j}(\frac{q^{\mp\frac{1}{2}}z_s}{w_s})^{k_s}\Bigr)X_{j}^{+}(\underline{w})\Psi_i^{(s)}(z_s)\\
&=&\Bigl (g_{ij}(\frac{z_s}{w_s}q^{\mp
\frac{1}{2}})\Bigr)^{\pm1}X_{j}^{\pm}(\underline{w})\Psi_i^{(s)}(z_s).
\end{eqnarray*}

For relation \eqref{n:tor7}, we have the  proposition as follows.
\begin{prop}
\begin{eqnarray*}
&&(z-q^{(\alpha_i|\alpha_j)}w)\,X_{i,s}^{\pm}(z)X_{j,s}^{\pm}(w)
=(q^{(\alpha_i|\alpha_j)}z-w)\,X_{j,s}^{\pm}(w)\,X_{i,s}^{\pm}(z).
\end{eqnarray*}
\end{prop}
\begin{proof}
Here we only check for the case of "+". The proof is divided into several cases as follows.

Case 1. $a_{ij}=0$, it is trivial.
Case 2. $(\alpha_i|\alpha_{i+1})=-\frac{1}{2}$, thanks to relations \eqref{7} and \eqref{11}, we have immediately,
\begin{gather*}
\begin{split}
 &(z-q^{-\frac{1}{2}}w)X_{i\epsilon,s}^+({z})X_{(i+1)\epsilon',s}^+({w})\\
 &=:X_{i\epsilon,s}^+ ({z})X_{(i+1)\epsilon',s}^+ ({w}):
   (z-q^{\epsilon'/2}w)^{-1}(z-q^{1/2}w)(z-q^{-\frac{1}{2}}w)\\
   &=: X_{(i+1)\epsilon',s}^+ ({w})X_{i\epsilon,s}^+ ({z}):(q^{\epsilon'/2}w-z)^{-1}(w-q^{1/2}z)(q^{-\frac{1}{2}}z-w)\\
 &  = X_{(i+1)\epsilon',s}^+ ({w})X_{i\epsilon,s}^+ ({z})(q^{-\frac{1}{2}}z-w),
\end{split}
\end{gather*}
which implies the assertion.

 Case 3. $(\alpha_i|\alpha_j)=-1$, that is $(i, j)$ is $(1, 0), (0, 1), (n-1, n)$ or$(n, n-1)$.
\begin{eqnarray*}
&&(z-q^{-1}w)X_{0\epsilon,s}^+({z})X_{1\epsilon',s}^+({w})\\
&=&:Z_{0,s}^+(q_\epsilon^{1+\epsilon}{z})Z_{0,s}^+(q_\epsilon^{-1+\epsilon}{z})Y_{0,s}^+({z}):
Z_{1,s}^+(q^{\epsilon'/2}{w})Z_{2,s}^-({w})Y_{1,s}^+({w})(z-q^{-1}w)\\
&=&:X_{0\epsilon,s}^+({z})X_{1\epsilon',s}^+({w}):(z-qw)(z-w)(z-q^{-1}w).
\end{eqnarray*}
On the other hand, it is easy to get that
\begin{eqnarray*}
&&(q^{-1}z -w) X_{1\epsilon',s}^+({w})X_{0\epsilon,s}^+({z})\\
&=&Z_{1,s}^+(q^{\epsilon'/2}{w})Z_{2,s}^-({w})Y_{1,s}^+({w}):Z_{0,s}^+(q_\epsilon^{1+\epsilon}{z})Z_{0,s}^+(q_\epsilon^{-1+\epsilon}{z})Y_{0,s}^+({z}):(q^{-1}z -w)\\
&=&X_{1\epsilon',s}^+({w})X_{0\epsilon,s}^+({z}):(w-qz)(w-z)(q^{-1}z -w).
\end{eqnarray*}

\indent  Case 4.  $(\alpha_i|\alpha_j)=2$, that is $i=j=0$ or n. Using\eqref{1} and \eqref{3} , it is clear to see that
\begin{eqnarray*}
&&(z-q^2w)X_{j\epsilon,s}^+({z})X_{j\epsilon',s}^+({w})\\
&=&(z-q^2w):Z_{j,s}^+(q_\epsilon^{1+\epsilon}{z})Z_{j,s}^+(q_\epsilon^{-1+\epsilon}{z})Y_{j,s}^+({z}):
:Z_{j,s}^+(q_{\epsilon'}^{1+\epsilon'}{w})Z_{j,s}^+(q_{\epsilon'}^{-1+\epsilon'}{w})Y_{j,s}^+({w}):\\
&=&:X_{j\epsilon,s}^+({z})X_{j\epsilon',s}^+({w}):\\
&&\times \dfrac{\prod\limits_{k,l=\pm 1}(q_\epsilon^{k+\epsilon}z-q_{\epsilon'}^{l+\epsilon'}w)
}{(z-w)(z-qw)(z-q^{-1}w)(z-q^2w)}(z-q^2w).\\
\end{eqnarray*}
In fact, similarly we have that
\begin{eqnarray*}
&&(q^2z - w)X_{j\epsilon',s}^+({w})X_{j\epsilon,s}^+({z})\\
&=&:Z_{j,s}^+(q_{\epsilon'}^{1+\epsilon'}{w})Z_{j,s}^+(q_{\epsilon'}^{-1+\epsilon'}{w})Y_{j,s}^+({w})::Z_{j,s}^+(q_\epsilon^{1+\epsilon}{z})Z_{j,s}^+(q_\epsilon^{-1+\epsilon}{z})Y_{j,s}^+({z}):(q^2z - w)\\
&=&:X_{j\epsilon',s}^+({w})X_{j\epsilon,s}^+({z}):\\
&&\times \dfrac{\prod\limits_{k,l=\pm 1}(q_\epsilon^{k+\epsilon}z-q_{\epsilon'}^{l+\epsilon'}w)
}{(w-z)(w-qz)(w-q^{-1}z)(w-q^2z)}(q^2z - w),\\
\end{eqnarray*}
which implies relation \eqref{n:tor7}.

 Case 5. $(\alpha_i|\alpha_i)=1$, that is, $i=1, \cdots, n-1$.  By relation \eqref{4}, one has that

\begin{eqnarray*}
&&(z-qw) X_{i\epsilon,s}^+({z})X_{i\epsilon',s}^+({w})\\
&=&:X_{i\epsilon,s}^+({z})X_{i\epsilon',s}^+({w}):
(q^{\epsilon/2}z-q^{\epsilon'/2}w)(z-qw)^{-1}(z-qw) \\
&=&:X_{i\epsilon',s}^+({w})X_{i\epsilon,s}^+({z}):(q^{\epsilon'/2}w -q^{\epsilon/2}z)(w -qz)^{-1}(qz-w)\\
&=&(qz-w)X_{i\epsilon',s}^+({w})X_{i\epsilon,s}^+({z}).
\end{eqnarray*}
Hence we have proved proposition 4.4.
\end{proof}

To check  relation \eqref{n:tor8}, we prove the following result. 
\begin{prop}\label{51} One has that
\begin{equation*}
[\,X_{i,s}^{+}(z),
X_{j,s}^{-}(w)\,]=\frac{\delta_{ij}}{(q-q^{-1})zw}\Big(\delta(zw^{-1}q^{\frac{1}{2}})\Phi_i^{(s)}(wq^{-\frac{1}{4}})
-\delta(zw^{-1}q^{-\frac{1}{2}})\Psi_i^{(s)}(zq^{-\frac{1}{4}})\Big)
\end{equation*}
\end{prop}
\begin{proof}
 We divided the proof into several cases.
 For the case of $(\alpha_i|\alpha_j)=-1$ such as $(i,j)=(1,0)$,  it follows from relations (4.9) and (4.15),
\begin{gather*}
\begin{split}
&(q^{-1}-q)(q^{1/2}-q^{-1/2})^2z^2w[X^+_{0,s}(z),X^-_{1,s}(w)]\\
&=\bigl(q^{1/2}X^+_{0+,s}(z)+q^{-1/2}X^+_{0-,s}(z)-(q^{1/2}+q^{-1/2})X^+_{00,s}(z)\bigr)\bigl(X^-_{1+, s}(w)-X^-_{1-,s}(w)\bigr)\\
&-\bigl(X^-_{1+, s}(w)-X^-_{1-,s}(w)\bigr)\bigl(q^{1/2}X^+_{0+,s}(z)+q^{-1/2}X^+_{0-,s}(z)-(q^{1/2}+q^{-1/2})X^+_{00,s}(z)\bigr)=0.
\end{split}
\end{gather*}

For the case of $j=0, n$  such that $a_{ij}=0$,  we have that
\begin{gather*}
\begin{split}
&X_{j\epsilon,s}^+({z})X_{i\epsilon',s}^-({w})=:Z_{j,s}^+(q_\epsilon^{1+\epsilon}{z})Z_{j,s}^+(q_\epsilon^{-1+\epsilon}{z})Y_{j,s}^+({z}):
Z_{i,s}^-({w})Z_{i+1,s}^+(q^{\epsilon'/2}{w})Y_{i,s}^-({w})\\
&=Z_{i,s}^-({w})Z_{i+1,s}^+(q^{\epsilon'/2}{w})Y_{i,s}^-({w}):Z_{j,s}^+(q_\epsilon^{1+\epsilon}{z})Z_{j,s}^+(q_\epsilon^{-1+\epsilon}{z})Y_{j,s}^+({z}):\\
&=X_{i\epsilon',s}^-({w})X_{j\epsilon,s}^+({z})
\end{split}
\end{gather*}
Similarly, it holds that
\begin{gather*}
\begin{split}
&X_{i\epsilon,s}^+({z}) X_{j,s}^-({w})=Z_{i,s}^+(q^{\frac{\epsilon}{2}}{z})Z_{i+1,s}^-({z})Y_{i,s}^+({z})
:Z_{j,s}^-(q^{\frac{1}{2}}{w})Z_{j,s}^-(q^{-\frac{1}{2}}{w}):Y_{j,s}^-({w}) \\&=
:Z_{j,s}^-(q^{1/2}{w})Z_{j,s}^-(q^{-1/2}{w}):Y_{j,s}^-({w})Z_{i,s}^+(q^{\epsilon/2}{z})Z_{i+1,s}^-({z})Y_{i,s}^+({z})\\
&=X_{j,s}^-({w})X_{i\epsilon,s}^+({z}).
\end{split}
\end{gather*}
 For the case of $(\alpha_i|\alpha_i)=1 $, that is, $i=1,...,n-1$. It follows from (4.5)-(4.6) that
\begin{gather*}
\begin{split}
 &[X_{i,s}^{+}(z), X_{i,s}^-(w)]\\
  &=\dfrac{1}{ (q^{\frac{1}{2}}-q^{-\frac{1}{2}})^2zw}\Big( :X_{i+,s}^+(z)X_{i-,s}^-(w): \dfrac{z-q^{\frac{1}{2}}w}{w}\delta (\dfrac{q^{-\frac{1}{2}}w}{z})
  \\&+ :X_{i+,s}^-(w)X_{i-,s}^+(z): \dfrac{z-q^{-\frac{1}{2}}w}{w}\delta (\dfrac{w}{q^{-\frac{1}{2}}z})  \Big)\\&
  =\dfrac{1}{ (q^{\frac{1}{2}}-q^{-\frac{1}{2}})zw} :X_{i+,s}^+(z)X_{i-,s}^-(w):\delta (\dfrac{q^{-\frac{1}{2}}w}{z})
  \\& \hspace{3cm}    +\dfrac{1}{ (q^{-\frac{1}{2}}-q^{\frac{1}{2}})zw} :X_{i+,s}^-(w)X_{i-,s}^+(z):\delta (\dfrac{w}{q^{-\frac{1}{2}}z})  \\
 &=\dfrac{1}{ (q^{\frac{1}{2}}-q^{-\frac{1}{2}})zw} :X_{i+,s}^+(q^{-\frac{1}{2}}w)X_{i-,s}^-(w):\delta (\dfrac{q^{-\frac{1}{2}}w}{z})
  \\& \hspace{3cm}    +\dfrac{1}{ (q^{-\frac{1}{2}}-q^{\frac{1}{2}})zw} :X_{i+,s}^-(w)X_{i-,s}^+(q^{\frac{1}{2}}w):\delta (\dfrac{w}{q^{-\frac{1}{2}}z}),
\end{split}
\end{gather*}
where we have used the property of the $\delta$-function: $$f(z_1,z_2)\delta(\frac{z_1}{z_2})=f(z_1,z_1)\delta(\frac{z_1}{z_2})=f(z_2,z_2)\delta(\frac{z_1}{z_2}).$$

Actually, we get by direct calculation
\begin{gather*}
\begin{split}
&:X_{i+,s}^+(q^{-\frac{1}{2}}w)X_{i-,s}^-(w):\\
&=:Z_{j,s}^+(q^{1/2}w)Z_{j,s}^+(q^{-1/2}w)Y_{j,s}^+(q^{-1/2}w)Z_{j,s}^-(q^{1/2}w)Z_{j,s}^-(q^{-1/2}w):Y_{j,s}^-(w):\\
&=\Phi_j^{(s)}(q^{-1/4}w).
\end{split}
\end{gather*}
In a similar manner, it is easy to see that
\begin{gather*}
\begin{split}
&:X_{i+,s}^-(w)X_{i-,s}^+(q^{\frac{1}{2}}w):\\
&=Z_{i,s}^-(w)Z_{i+1,s}^+(q^{1/2}w)Y_{i,s}^-(w)Z_{i,s}^+(w)Z_{i+1,s}^-(q^{\frac{1}{2}}w)Y_{i,s}^+(q^{\frac{1}{2}}w)\\
&=\Psi_i^{(s)}(wq^{1/4}).
\end{split}
\end{gather*}

Inserting the above expressions into the left hand side of proposition 4.5, we get that
\begin{gather*}
\begin{split}
 &[X_{i,s}^{+}(z),X_{i,s}^-(w)]\\
 & =\dfrac{1}{
  (q^{1/2}-q^{-1/2})zw}
  \times \Big(\Phi_i^{(s)}(wq^{-1/4})\delta(\dfrac {q^{-1/2}w}{z})-\Psi_i^{(s)}(wq^{1/4})\delta (\dfrac{w}{q^{-1/2}z}) \Big).
\end{split}
\end{gather*}

Lastly, we need consider the case of $(\alpha_j|\alpha_j)=2 $ for $j=0,n$.
 In this case we use relations \eqref{12} and \eqref{13},  for $\epsilon=0$ , we get $ X_{j0,s}^+({z})X_{j,s}^-({w})=:X_{j0,s}^+({z})X_{j,s}^-({w}):=X_{j,s}^-({w})X_{j0,s}^+({z})$,
so we have that
\begin{gather*}
\begin{split}
&[X_{j,s}^+(z),X_{j,s}^{-}(w)]\\
 &= \dfrac{-1}{(q-q^{-1})^2z^2}\sum_{\epsilon=\pm 1} q^{\epsilon/2}:X^+_{j\epsilon,s}(z)X^-_{j,s}(w): \dfrac{q^{-3\epsilon/2}z-w}{w}\Big( \dfrac{1}{q^{\epsilon/2}z/w-1}-\dfrac{1}{1-q^{\epsilon/2}z/w}    \Big)                             \\
 &=\dfrac{-1}{(q-q^{-1})^2z^2}\sum_{\epsilon=\pm 1} :X^+_{j\epsilon,s}(z)X^-_{j,s}(w):  (q^{-3\epsilon/2}-q^{\epsilon/2})\delta(\dfrac{q^{\epsilon/2}z}{w})\\
&=\dfrac{1}{(q-q^{-1})zw}\Big(\Phi^{(s)}_j(wq^{-1/4})\delta(\dfrac{q^{1/2}z}{w})-\Psi^{(s)}_j(wq^{1/4})\delta(\dfrac{zq^{-1/2}}{w})   \Big),
\end{split}
\end{gather*}
where we have used

\begin{gather*}
\begin{split}
&\dfrac{q^{1/2}-q^{-3/2}}{(q-q^{-1})^2z^2}:X^+_{j\epsilon,s}(z)X^-_{j,s}(w):\\
&=\dfrac{q^{-1/2}}{(q-q^{-1})z^2}:X^+_{j\epsilon,s}(q^{-1/2}w)X^-_{j,s}(w):\\
&=\dfrac{q^{-1/2}}{(q-q^{-1})z^2}:Z_{j,s}^+(q^{1/2}w)Z_{j,s}^+(q^{-1/2}w)Y_{j,s}^+(q^{-1/2}w)Z_{j,s}^-(q^{1/2}w)Z_{j,s}^-(q^{-1/2}w):Y_{j,s}^-(w):\\
&=\dfrac{1}{(q-q^{-1})zw}\exp\Big(\sum_{k=1}^{\infty}a_{j}^{(s)}(k)(q-q^{-1})(q^{-1/4}w)^{-k}     \Big)q^{a_j(0)}\\
&=\dfrac{1}{(q-q^{-1})zw}\Phi_j^{(s)}(q^{-1/4}w),
\end{split}
\end{gather*}
and
\begin{gather*}
\begin{split}
&\dfrac{-1}{(q-q^{-1})zw}:X_{j-,s}^+(q^{1/2}w)X^-_{j,s}(w):\\
&=\dfrac{-1}{(q-q^{-1})zw}:Z_{j,s}^+(q^{1/2}w)Z_{j,s}^+(q^{-1/2}w)Y_{j,s}^+(q^{1/2}w)::Z_{j,s}^-(q^{1/2}w)Z_{j,s}^-(q^{-1/2}w):Y_{j,s}^-(w)\\
&=\dfrac{-1}{(q-q^{-1})zw}\exp\Big( \sum_{k=1}^{\infty}a_{j}^{(s)}(-k)(-1)(q-q^{-1})(q^{1/4}w)^{k} \Big)q^{-a_j(0)}\\
&=\dfrac{-1}{(q-q^{-1})zw}\Psi_j^{(s)}(q^{1/4}w).
\end{split}
\end{gather*}
\end{proof}

For the Serre relation \eqref{n:tor9}, we have the following proposition. 
\begin{prop} \label{52} For $i\neq j$
\begin{eqnarray*}
  &&\text{Sym}_{{z_1},\cdots
		{z_{n}}}\sum_{k=0}^{n=1-a_{i{j}}}(-1)^k
	\Big[{n\atop  k}\Big]_{i}X_{i,s}^{\pm}({z_1})\cdots X_{i,s}^{\pm}({z_k}) X_{j,s}^{\pm}({w})
X_{i,s}^{\pm}({z_{k+1}})\cdots X_{i,s}^{\pm}({z_{n}})=0.
\end{eqnarray*}
\end{prop}

\begin{proof}
\, First let us see the case when $a_{ij}=-1$ for $i=1,...n-2$. Here we only check it for "+", it is similar for the case "-". We list the relations that will be used,
   \begin{multline*}
     X^+_{i\epsilon_1,s}({z_1})X^+_{i\epsilon_2,s}({z_2})X^+_{(i+1)\epsilon,s}({w})\\
    = :X^+_{i\epsilon_1,s}({z_1})X^+_{i\epsilon_2,s}({z_2})X^+_{(i+1)\epsilon,s}({w}):
    \dfrac{(q^{\epsilon_1/2}z_{1}-q^{\epsilon_2/2}z_{2})(z_{1}-q^{1/2}w)(z_{2}-q^{1/2}w)}{(z_{1}-qz_{2})
    (z_{1}-q^{\epsilon/2}w)(z_{2}-q^{\epsilon/2}w)},
   \end{multline*}

\begin{multline*}
     X^+_{i\epsilon_1,s}({z_1})X^+_{(i+1)\epsilon,s}({w})X^+_{i\epsilon_2,s}({z_2})\\
    = :X^+_{i\epsilon_1,s}({z_1})X^+_{(i+1)\epsilon,s}({w})X^+_{i\epsilon_2,s}({z_2}):
      \dfrac{(z_{1}-q^{1/2}w)(q^{\epsilon_1/2}z_{1}-q^{\epsilon_2/2}z_{2})(q^{1/2}z_{2}-w)}
      {(z_{1}-qz_{2})(z_{1}-q^{\epsilon/2}w)(z_{2}-q^{\epsilon/2}w)},
  \end{multline*}
\begin{multline*}
   X^+_{(i+1)\epsilon,s}({w})X^+_{i\epsilon_1,s}({z_1})X^+_{i\epsilon_2,s}({z_2})\\
    = :X^+_{(i+1)\epsilon,s}({w})X^+_{i\epsilon_1,s}({z_1})X^+_{i\epsilon_2,s}({z_2}):
    \dfrac{(w-q^{1/2}z_{1})(w-q^{1/2}z_{2})(q^{\epsilon_1/2}z_{1}-q^{\epsilon_2/2}z_{2})}{(z_{1}-qz_{2})
    (z_{1}-q^{\epsilon/2}w)(z_{2}-q^{\epsilon/2}w)}.
\end{multline*}
Thus we get that
\begin{gather*}
\begin{split}
 &X^+_{i\epsilon_1,s}({z_1})X^+_{i\epsilon_2,s}({z_2})X^+_{(i+1)\epsilon,s}({w})-
 (q^{\frac{1}{2}}+q^{-\frac{1}{2}})X^+_{i\epsilon_1,s}({z_1})X^+_{(i+1)\epsilon,s}({w})X^+_{i\epsilon_2,s}({z_2})
 \\
 & \hspace{2cm} +X^+_{(i+1)\epsilon,s}({w})X^+_{i\epsilon_1,s}({z_1})X^+_{i\epsilon_2,s}({z_2})\\
 &=:X^+_{i\epsilon_1,s}({z_1})X^+_{i\epsilon_2,s}({z_2})X^+_{(i+1)\epsilon,s}({w}):\dfrac{q^{\epsilon_1/2}z_{1}-q^{\epsilon_2/2}z_{2}}{(z_{1}-qz_{2})(z_{1}-q^{\epsilon/2}w)(z_{2}-q^{\epsilon/2}w)}
    \\ &\times \big((z_{1}-q^{1/2}w)(z_{2}-q^{1/2}w)+ (q^{1/2}+q^{-1/2})(z_{1}-q^{1/2}w)(w-q^{1/2}z_{2})\\
    &\hspace{3cm}   +(w-q^{1/2}z_{1})(w-q^{1/2}z_{2})\big)\\
     & =:X^+_{i\epsilon_1,s}({z_1})X^+_{i\epsilon_2,s}({z_2})X^+_{(i+1)\epsilon,s}({w}):
  \dfrac{(q^{\epsilon_1/2}z_{1}-q^{\epsilon_2/2}z_{2})(q^{-1/2}-q^{1/2})w}{(z_{1}-q^{\epsilon/2}w)(z_{2}-q^{\epsilon/2}w)}.
  \end{split}
\end{gather*}

We can see that the last part is antisymmetric under $({z_1},\epsilon_1)\mapsto ({z_2},\epsilon_2)$, which implies the Serre relation
$ X^+_{i,s}({z_1})X^+_{i,s}({z_2})X^+_{i+1,s}({w})+...+({z_1}\leftrightarrow {z_2})=0 $.
For the same reason, it is easy to see that $$X^+_{i,s}({z_1})X^+_{i,s}({z_2})X^+_{i-1,s}({w})-(q^{1/2}+q^{-1/2})X^+_{i,s}({z_1})X^+_{i-1,s}({w})X^+_{i,s}({z_2})
+\cdots=0.$$

 In the case of $a_{ij}=-1$ for $i=0,j=1$ or $i=n,j=n-1$.
We only  show the proof in $-$ cases.  First, for $i=0$, $j=1$,    in this situation we have the following relations,
\begin{multline*}
X^-_{0,s}({z_1})X^-_{0,s}({z_2})X^-_{1\epsilon,s}({w})=:X^-_{0,s}({z_1})X^-_{0,s}({z_2})X^-_{1\epsilon,s}({w}):\\
  \times  \dfrac{(z_{1}-z_{2})(z_{1}-q^{-1}w)(z_{2}-q^{-1}w)(z_1-w)(z_2-w)}{(z_1-q^{-2}z_2)},    \hspace{1.6cm}
\end{multline*}
\begin{multline*}
X^-_{0,s}({z_1})X^-_{1\epsilon,s}({w})X^-_{0,s}({z_2})=:X^-_{0,s}({z_1})X^-_{1\epsilon,s}({w})X^-_{0,s}({z_2}):\\
  \times  \dfrac{(z_{1}-z_{2})(z_{1}-q^{-1}w)(z_1-w)(q^{-1}z_{2}-w)(z_2-w)}{(z_1-q^{-2}z_2)},    \hspace{1.6cm}
\end{multline*}
\begin{multline*}
X^-_{1\epsilon,s}({w})X^-_{0,s}({z_1})X^-_{0,s}({z_2})=:X^-_{1\epsilon,s}({w})X^-_{0,s}({z_1})X^-_{0,s}({z_2}):\\
  \times  \dfrac{(z_{1}-z_{2})(q^{-1}z_{1}-w)(z_1-w)(q^{-1}z_{2}-w)(z_2-w)}{(z_1-q^{-2}z_2)}.    \hspace{1.6cm}
\end{multline*}
     Using the above expressions, we get that,
 \begin{eqnarray*}
\begin{split}
 &X^-_{0,s}({z_1})X^-_{0,s}({z_2})X^-_{1\epsilon,s}({w})-(q+q^{-1})X^-_{0,s}({z_1})X^-_{1\epsilon,s}({w})X^-_{0,s}({z_2})  +X^-_{1\epsilon,s}({w})X^-_{0,s}({z_1})X^-_{0,s}({z_2})\\
 &=:X^-_{0,s}({z_1})X^-_{0,s}({z_2})X^-_{1\epsilon,s}({w}):\dfrac{z_{1}-z_{2}}{z_{1}-q^{-2}z_{2}}
 \Big((z_{1}-q^{-1}w)(z_{2}-q^{-1}w)(z_1-w)(z_2-w)
 \\&  \hspace{2cm}   -(q+q^{-1})(z_{1}-q^{-1}w)(z_1-w)(q^{-1}z_{2}-w)(z_2-w)
 \\& \hspace{3cm} +(q^{-1}z_{1}-w)(z_1-w)(q^{-1}z_{2}-w)(z_2-w) \Big)\\
&=:X^-_{0,s}({z_1})X^-_{0,s}({z_2})X^-_{1\epsilon,s}({w}):(q+q^{-1})w(z_{1}-z_{2})(z_1-w)(z_2-w).
\end{split}
\end{eqnarray*}
Obviously the antisymmetry with regard to $(z_1\leftrightarrow z_2)$ implies the case.

Second,  for $i=n$, $j=n-1$,    similarly we have the following relations,

\begin{multline*}
X^-_{n,s}({z_1})X^-_{n,s}({z_2})X^-_{(n-1)\epsilon,s}({w})=:X^-_{n,s}({z_1})X^-_{n,s}({z_2})X^-_{(n-1)\epsilon,s}({w}):\\
  \times  \dfrac{(z_{1}-z_{2})(z_{1}-q^{-1}w)(z_{2}-q^{-1}w)}{(z_{1}-q^{-2}z_{2})(z_{1}-q^\epsilon w)(z_{2}-q^\epsilon w)},    \hspace{1.6cm}
\end{multline*}
\begin{multline*}
   X^-_{n,s}({z_1})X^-_{(n-1)\epsilon,s}({w})X^-_{n,s}({z_2})=:X^-_{n,s}({z_1})X^-_{(n-1)\epsilon,s}({w})X^-_{n,s}({z_2}):\\
   \times\dfrac{z_{1}-z_{2}}{q(z_{1}-q^{-2}z_{2})}\Big( \dfrac{z_{1}-q^{-1}w}{z_{1}-qw}\Big)^{\frac{1+\epsilon}{2}}\Big( \dfrac{z_{2}-qw}{z_{2}-q^{-1}w}\Big)^{\frac{1-\epsilon}{2}},
\end{multline*}
\begin{multline*}
   X^-_{(n-1)\epsilon,s}({w})X^-_{n,s}({z_1})X^-_{n,s}({z_2})=:X^-_{(n-1)\epsilon,s}({w})X^-_{n,s}({z_1})X^-_{n,s}({z_2}):\\   \times\dfrac{z_{1}-z_{2}}{q^{\epsilon2}(z_{1}-q^{-2}z_{2})}\Big(\dfrac{(w-q^{-1}z_{1})(w-q^{-1}z_{2})}{(w-qz_{1})(w-qz_{2})}  \Big)^{\frac{1-\epsilon}{2}}.
\end{multline*}

 It is easy to get that
\begin{eqnarray*}
\begin{split}
 &X^-_{n,s}({z_1})X^-_{n,s}({z_2})X^-_{(n-1)+,s}({w})-(q+q^{-1})X^-_{n,s}({z_1})X^-_{(n-1)+,s}({w})X^-_{n,s}({z_2})\\
 &\hspace{2cm}  +X^-_{(n-1)+,s}({w})X^-_{n,s}({z_1})X^-_{n,s}({z_2})\\
 &=:X^-_{n,s}({z_1})X^-_{n,s}({z_2})X^-_{(n-1)+,s}({w}):\dfrac{z_{1}-z_{2}}{z_{1}-q^{-2}z_{2}}\\
  &\times\Big(\dfrac{(z_{1}-q^{-1}w)(z_{2}-q^{-1}w)}{(z_{1}-qw)(z_{2}-qw)} -(1+q^{-2})\dfrac{z_{1}-q^{-1}w}{z_{1}-qw}+q^{-2} \Big)\\
  & =:X^-_{n,s}({z_1})X^-_{n,s}({z_2})X^-_{(n-1)+,s}({w}):\dfrac{z_{1}-z_{2}}{(z_{1}-qw)(z_{2}-qw)}.
\end{split}
\end{eqnarray*}
\begin{eqnarray*}
\begin{split}
 &X^-_{n,s}(z_1)X^-_{n,s}(z_2)X^-_{(n-1)-,s}(w)-(q+q^{-1})X^-_{n,s}(z_1)X^-_{(n-1)-,s}(w)X^-_{n,s}(z_2)\\
 &\hspace{2cm}  +X^-_{(n-1)-,s}(w)X^-_{n,s}(z_1)X^-_{n,s}(z_2)\\&
  =:X^-_{n,s}(z_1)X^-_{n,s}(z_2)X^-_{(n-1)-,s}(w):\dfrac{z_{1}-z_{2}}{z_{1}-q^{-2}z_{2}}
  \\&\times\Big(1- (1+q^{-2})\dfrac{z_{2}-qw}{z_{2}-q^{-1}w}+q^2\dfrac{(w-q^{-1}z_{1})(w-q^{-1}z_{2})}{(w-qz_{1})(w-qz_{2})}   \Big)\\&
 =:X^-_{n,s}(z_1)X^-_{n,s}(z_2)X^-_{(n-1)-,s}(w):\dfrac{(z_{1}-z_{2})w(q-q^{-1})}{(z_{1}-q^{-1}w)(z_{2}-q^{-1}w)}.
\end{split}
\end{eqnarray*}

So we get the conclusion through the antisymmetry of the above two expressions.

In the case of $a_{ij}=-2$ for $i=1,j=0$ and $i=n-1,j=n$ ,  here we only give the proof for  $i=n-1, j=n$. First we have that,
 \begin{multline*}
 X^+_{(n-1)\epsilon_1,s}({z_1})X^+_{(n-1)\epsilon_2,s}({z_2})X^+_{(n-1)\epsilon_3,s}({z_3})X^+_{n\epsilon,s}({w})\\
\hspace{3cm}  =:X^+_{(n-1)\epsilon_1,s}({z_1})X^+_{(n-1)\epsilon_2,s}({z_2})X^+_{(n-1)\epsilon_3,s}({z_3})X^+_{n\epsilon,s}({w}):\\
 \hspace{2cm}  \times\prod \limits_{1\leq k <t\leq 3}^{}\dfrac{q^{\epsilon_k/2}z_{k}-q^{\epsilon_t/2}z_{t}}{z_{k}-qz_{t}}\prod\limits_{k=1}^{3}\dfrac{(z_{k}-qw)(z_{k}-w)}{(z_{k}-q^{1+\epsilon}_\epsilon w)(z_{k}-q^{-1+\epsilon}_\epsilon w)}
\end{multline*}

 \begin{multline*}
 X^+_{(n-1)\epsilon_1,s}({z_1})X^+_{(n-1)\epsilon_2,s}({z_2})X^+_{n\epsilon,s}({w})X^+_{(n-1)\epsilon_3,s}({z_3})\\ \hspace{3cm}
=:X^+_{(n-1)\epsilon_1,s}({z_1})X^+_{(n-1)\epsilon_2,s}({z_2})X^+_{n\epsilon,s}({w})X^+_{(n-1)\epsilon_3,s}({z_3}): \\\times \prod \limits_{1\leq k <t\leq 3}^{}\dfrac{q^{\epsilon_k/2}z_{k}-q^{\epsilon_t/2}z_{t}}{z_{k}-qz_{t}}\prod\limits_{k=1}^{3}\dfrac{(z_{k}-w)}{(z_{k}-q^{1+\epsilon}_\epsilon w)(z_{k}-q^{-1+\epsilon}_\epsilon w)}
 (z_{1}-qw)(z_{2}-qw)(qz_{3}-w)
 \end{multline*}
  \begin{multline*}
X^+_{(n-1)\epsilon_1,s}({z_1})X^+_{n\epsilon,s}({w})X^+_{(n-1)\epsilon_2,s}({z_2})X^+_{(n-1)\epsilon_3,s}({z_3})\\ \hspace{3cm}
=:X^+_{(n-1)\epsilon_1,s}({z_1})X^+_{n\epsilon,s}({w})X^+_{(n-1)\epsilon_2,s}({z_2})X^+_{(n-1)\epsilon_3,s}({z_3}): \\\times \prod \limits_{1\leq k <t\leq 3}^{}\dfrac{q^{\epsilon_k/2}z_{k}-q^{\epsilon_t/2}z_{t}}{z_{k}-qz_{t}}\prod\limits_{k=1}^{3}\dfrac{(z_{k}-w)}{(z_{k}-q^{1+\epsilon}_\epsilon w)(z_{k}-q^{-1+\epsilon}_\epsilon w)} (z_{1}-qw)(w-qz_{2})(w-qz_{3})
 \end{multline*}
  \begin{multline*}
X^+_{n\epsilon,s}({w})X^+_{(n-1)\epsilon_1,s}({z_1})X^+_{(n-1)\epsilon_2,s}({z_2})X^+_{(n-1)\epsilon_3,s}({z_3})   \\ \hspace{3cm}
=:X^+_{n\epsilon,s}({w})X^+_{(n-1)\epsilon_1,s}({z_1})X^+_{(n-1)\epsilon_2,s}({z_2})X^+_{(n-1)\epsilon_3,s}({z_3}): \\\times \prod \limits_{1\leq k <t\leq 3}^{}\dfrac{q^{\epsilon_k/2}z_{k}-q^{\epsilon_t/2}z_{t}}{z_{k}-qz_{t}}\prod\limits_{k=1}^{3}\dfrac{(z_{k}-w)(qz_{k}-w)}{(z_{k}-q^{1+\epsilon}_\epsilon w)(z_{k}-q^{-1+\epsilon}_\epsilon w)}.  \hspace{1cm}   
\end{multline*}
Thus one gets that,
\begin{eqnarray*}
&&X^+_{(n-1)\epsilon_1,s}({z_1})X^+_{(n-1)\epsilon_2,s}({z_2})X^+_{(n-1)\epsilon_3,s}({z_3})X^+_{n\epsilon,s}({w})
\\&& \hspace{1cm} -[3]_{q^{1/2}}X^+_{(n-1)\epsilon_1,s}({z_1})X^+_{(n-1)\epsilon_2,s}({z_2})X^+_{n\epsilon,s}({w})X^+_{(n-1)\epsilon_3,s}({z_3})\\
&&\hspace{1cm} +[3]_{q^{1/2}}X^+_{(n-1)\epsilon_1,s}({z_1})X^+_{(n-1)\epsilon_2,s}({z_2})X^+_{n\epsilon,s}({w})X^+_{(n-1)\epsilon_3,s}({z_3})\\
&&\hspace{1cm}-X^+_{n\epsilon,s}({w})X^+_{(n-1)\epsilon_1,s}({z_1})X^+_{(n-1)\epsilon_2,s}({z_2})X^+_{(n-1)\epsilon_3,s}({z_3})\\&&
=:X^+_{(n-1)\epsilon_1,s}({z_1})X^+_{(n-1)\epsilon_2,s}({z_2})X^+_{(n-1)\epsilon_3,s}({z_3})X^+_{n\epsilon,s}({w}):
\\&& \hspace{2cm}  \times \prod \limits_{1\leq k <t\leq 3}^{}\dfrac{q^{\epsilon_k/2}z_{k}-q^{\epsilon_t/2}z_{t}}{z_{k}-qz_{t}}\prod\limits_{k=1}^{3}\dfrac{(z_{k}-w)}{(z_{k}-q^{1+\epsilon}_\epsilon w)(z_{k}-q^{-1+\epsilon}_\epsilon w)}\\&&
\times \big((z_{1}-qw)(z_{2}-qw)(z_{3}-qw)+(q+1+q^{-1})(z_{1}-qw)(z_{2}-qw)(w-qz_{3})\\&&
+(q+1+q^{-1})(z_{1}-qw)(w-qz_{2})(w-qz_{3})+(w-qz_{1})(w-qz_{2})(w-qz_{3})\big)\\&&
=:X^+_{(n-1)\epsilon_1,s}({z_1})X^+_{(n-1)\epsilon_2,s}({z_2})X^+_{(n-1)\epsilon_3,s}({z_3})X^+_{n\epsilon,s}({w}):\\
&& \hspace{2cm}   \times \prod \limits_{1\leq k <t\leq 3}^{}\dfrac{q^{\epsilon_k/2}z_{k}-q^{\epsilon_t/2}z_{t}}{z_{k}-qz_{t}}\prod\limits_{k=1}^{3}\dfrac{(z_{k}-w)}{(z_{k}-q^{1+\epsilon}_\epsilon w)(z_{k}-q^{-1+\epsilon}_\epsilon w)}\\&&  \times
(q^{-1}-q)\big(w^2(z_{1}-(q^2+q)z_{2}+q^3z_{3})+w(z_{1}z_{2}-(q^2+q)z_{1}z_{3}+q^3z_{2}z_{3})  \big) \qquad (*)
\end{eqnarray*}

 Now we rewrite the equation $(*)$ into this to separate the symmetric part,
 \begin{eqnarray*}
 &&X^+_{(n-1)\epsilon_1,s}({z_1})X^+_{(n-1)\epsilon_2,s}({z_2})X^+_{(n-1)\epsilon_3,s}({z_3})X^+_{n\epsilon,s}({w})\\
 && \hspace{1cm} -[3]_{q^{1/2}}X^+_{(n-1)\epsilon_1,s}({z_1})X^+_{(n-1)\epsilon_2,s}({z_2})X^+_{n\epsilon,s}({w})X^+_{(n-1)\epsilon_3,s}({z_3})\\&& \hspace{1cm}
+[3]_{q^{1/2}}X^+_{(n-1)\epsilon_1,s}({z_1})X^+_{n\epsilon,s}({w})X^+_{(n-1)\epsilon_2,s}({z_2})X^+_{(n-1)\epsilon_3,s}({z_3})\\
&& \hspace{1cm} -X^+_{n\epsilon,s}({w})X^+_{(n-1)\epsilon_1,s}({z_1})X^+_{(n-1)\epsilon_2,s}({z_2})X^+_{(n-1)\epsilon_3,s}({z_3})\\
&&
=:X^+_{(n-1)\epsilon_1,s}({z_1})X^+_{(n-1)\epsilon_2,s}({z_2})X^+_{(n-1)\epsilon_3,s}({z_3})X^+_{n\epsilon,s}({w}):
\\&&\times \prod \limits_{1\leq k <t\leq 3}^{}\dfrac{(q^{\epsilon_k/2}z_{k}-q^{\epsilon_t/2}z_{t})}{(z_{k}-qz_{t})(qz_{k}-z_{t})}\prod\limits_{k=1}^{3}\dfrac{(z_{k}-w)}{(z_{k}-q^{1+\epsilon}_\epsilon w)(z_{k}-q^{-1+\epsilon}_\epsilon w)}\\&& \times
(q^{-1}-q)\big(w^2(z_{1}-(q^2+q)z_{2}+q^3z_{3})+w(z_{1}z_{2}-(q^2+q)z_{1}z_{3}+q^3z_{2}z_{3})  \big)\\
&& \hspace{8cm} \times \prod \limits_{1\leq k <t\leq 3}^{} (qz_{k}-z_{t}).
\end{eqnarray*}

We can see that the proposition holds as long as the following equation is correct in this case,
\begin{eqnarray*}\label{50}
  &&\sum\limits_{\sigma \in S_3}^{}sgn(\sigma)\sigma .\big( (x_{1}-(q^2+q)x_{2}+q^3x_{3}) \prod \limits_{1\leq k <t\leq 3}^{} (qz_{k}-z_{t}) \big) \\
 && =(q^2-1)^2(1+q+q^2)(x_3z_1z_2(z_1-z_2)+z_3(x_1z_2(z_2-z_3)+x_2z_1(z_3-z_1)))
\end{eqnarray*}
from which we immediately get that,
\begin{eqnarray*}
  \sum\limits_{\sigma \in S_3}^{}sgn(\sigma)\sigma .\big( (z_{1}-(q^2+q)z_{2}+q^3z_{3}) \prod \limits_{1\leq k <t\leq 3}^{} (qz_{k}-z_{t}) \big)=0,
\end{eqnarray*}
where $\sigma\in S_3$ acts on $z_{is}$ in the natural way, $\sigma .z_{is}=z_{\sigma(i)s}$. At the same time, it follows that,
\begin{eqnarray*}
  \sum\limits_{\sigma \in S_3}^{}sgn(\sigma)\sigma .\big( (z_{1}z_2-(q^2+q)z_1z_{3}+q^3z_2z_{3}) \prod \limits_{1\leq k <t\leq 3}^{} (qz_{k}-z_{t}) \big)=0.
\end{eqnarray*}
\end{proof}

Finally, we are left to verify relation \eqref{n:tor10}, which is special for the quantum $N$-toroidal case.
Without loss of generality, we take $i=0$ and $-$ for example, others cases can be checked similarly.

It is easy to see that for $\epsilon=\pm, 0$ and $s\neq s'$, $j=0, n$
\begin{eqnarray*}
&&X^-_{j,s}(z)X^+_{j\epsilon, s'}(w)=:X^-_{j,s}(z)X^+_{j\epsilon, s'}(w):z^{-12}q,\\
&&X^+_{j\epsilon, s'}(w)X^-_{js}(z)=:X^-_{js}(z)X^+_{j\epsilon, s'}(w):w^{-12}q^{-6\epsilon}.
\end{eqnarray*}


Subsequently by Lemma \ref{lemm1} it follows that 
\begin{eqnarray*}
&&{X}^-_{j,s}(z_1){X}^-_{j,s}(z_2){X}^-_{j,s}(z_3){X}^+_{j,s'}(w)\\
&&=-q^3\frac{(q^{1/2}:X_+:+q^{-1/2}:X_-:-[2]_{q^{1/2}}:X_0:)}{(q-q^{-1})(q^{1/2}-q^{-1/2})w^2(z_1z_2z_3)^{12}}\prod_{1\leq i<k\leq 3}\frac{z_i-z_k}{z_i-q^{-2}z_k},
\end{eqnarray*}
\begin{eqnarray*}
&&{X}^-_{j,s}(z_1){X}^-_{j,s}(z_2){X}^+_{j,s'}(w){X}^-_{j,s}(z_3)\\
&&=-q^2\frac{(q^{1/2-6}:X_+:+q^{-1/2+6}:X_-:-[2]_{q^{1/2}}:X_0:)}{(q-q^{-1})(q^{1/2}-q^{-1/2})w^2(z_1z_2w)^{12}}\prod_{1\leq i<k\leq 3}\frac{z_i-z_k}{z_i-q^{-2}z_k},
\end{eqnarray*}

\begin{eqnarray*}
&&{X}^-_{j,s}(z_1){X}^+_{j,s'}(w){X}^-_{j,s}(z_2){X}^-_{j,s}(z_3)\\
&&=-q\frac{(q^{1/2-12}:X_+:+q^{-1/2+12}:X_-:-[2]_{q^{1/2}}:X_0:)}{(q-q^{-1})(q^{1/2}-q^{-1/2})w^2(z_1w^2)^{12}}\prod_{1\leq i<k\leq 3}\frac{z_i-z_k}{z_i-q^{-2}z_k},
\end{eqnarray*}

\begin{eqnarray*}
&&{X}^+_{j,s'}(w){X}^-_{j,s}(z_1){X}^-_{j,s}(z_2){X}^-_{j,s}(z_3)\\
&&=-\frac{(q^{1/2-18}:X_+:+q^{-1/2+18}:X_-:-[2]_{q^{1/2}}:X_0:)}{(q-q^{-1})(q^{1/2}-q^{-1/2})w^2w^{36}}\prod_{1\leq i<k\leq 3}\frac{z_i-z_k}{z_i-q^{-2}z_k},
\end{eqnarray*}
where $:X_{\epsilon}:=:{X}^+_{j\epsilon,s'}(w){X}^-_{j,s}(z_1){X}^-_{j,s}(z_2){X}^-_{j,s}(z_3):$.

Up to a general factor $(*)$, we see that 
\begin{eqnarray*}
&&\sum_{k=0}^{3}(-1)^k
	\Big[{3\atop  k}\Big]_{0}{X}_{j,s}^{-}(z_1)\cdots {X}_{j,s}^{-}(z_{k}){X}_{j,s'}^{+}(w) {X}_{0,s}^{-}(z_{k+1})\cdots{X}_{0,s}^{-}(z_{3})\\
 &=&(*)\Bigl(q^{\frac{1}{2}}:X_{+}:\bigl((z_1z_2z_3)^{-12}-[3]_0q^{-6}(z_1z_2w)^{-12}+[3]_0q^{-12}(z_1w^2)^{-12}+q^{-18}w^{-36}\bigr)\\
 &&\hskip0.5cm +q^{-\frac{1}{2}}:X_{-}:\bigl((z_1z_2z_3)^{-12}-[3]_0q^{6}(z_1z_2w)^{-12}+[3]_0q^{12}(z_1w^2)^{-12}+q^{18}w^{-36}\bigr)\\
 &&\hskip0.5cm -[2]_{q^{1/2}}:X_{0}:\bigl((z_1z_2z_3)^{-12}-[3]_0(z_1z_2w)^{-12}+[3]_0(z_1w^2)^{-12}+w^{-36}\bigr)\Bigr),
 \end{eqnarray*}
where $*=\frac{1}{(q-q^{-1})(q^{1/2}-q^{-1/2})w^2}\prod_{1\leq i<k\leq 3}\frac{z_i-z_k}{z_i-q^{-2}z_k}$.

Successively taking limit with respect to $z_1\to w, z_2\to w, z_3\to w$ imply that
\begin{eqnarray*}
\lim_{z_i\to w}\sum_{k=0}^{3}(-1)^k
	\Big[{3\atop  k}\Big]_{0}{X}_{0,s}^{-}(z_1)\cdots {X}_{0,s}^{-}(z_{k}){X}_{0,s'}^{+}(w) {X}_{0,s}^{-}(z_{k+1})\cdots{X}_{0,s}^{-}(z_{3})=0.
\end{eqnarray*}

 Thus we have completed the proof of Theorem 3.1.

\vskip20pt \centerline{\bf ACKNOWLEDGMENT}

N. Jing would like to thank the support of
Simons Foundation grant 523868 and NSFC grant 11531004.
H. Zhang would
like to thank the support of NSFC grant 11871325.
\bigskip

\bibliographystyle{amsalpha}

\end{document}